\documentclass[oneside,english]{amsart}
\pdfoutput=1
\usepackage[T1]{fontenc}
\usepackage[latin9]{inputenc}
\usepackage{geometry}
\usepackage{babel}
\usepackage{mathtools}
\usepackage{amsbsy}
\usepackage{amstext}
\usepackage{amsthm}
\usepackage{amssymb}
\usepackage{graphicx}
\usepackage{verbatim}
\usepackage[unicode=true,pdfusetitle, bookmarks=true,bookmarksnumbered=false,bookmarksopen=false, breaklinks=false,pdfborder={0 0 0},backref=false] {hyperref}
\usepackage{cleveref}
\usepackage[all]{xy}
\usepackage{xcolor}
\usepackage{caption}
\usepackage{soul}
\usepackage{multirow}
\usepackage{todonotes}
\makeatletter

\numberwithin{equation}{section}
\numberwithin{figure}{section}
\theoremstyle{plain}
\newtheorem{thm}{\protect\theoremname}[section]

\newtheorem{lem}[thm]{Lemma}

\theoremstyle{definition}

\newtheorem{rem}[thm]{\protect\remarkname}

\author{Irving Dai}
\address{Department of Mathematics\\The University of Texas at Austin, Austin, USA}
\email{irving.dai@math.utexas.edu}
\author{Sungkyung Kang}
\address{Mathematical Institute, University of Oxford, Oxford, United Kingdom}
\email{sungkyung38@icloud.com}
\author{Abhishek Mallick}
\address{Department of Mathematics\\Rutgers University--New Brunswick, New Brunswick, USA}
\email{abhishek.mallick@rutgers.edu}
\author{JungHwan Park}
\address{Department of Mathematical Sciences, Korea Advanced Institute for Science and Technology, Daejeon, Republic of Korea}
\email{jungpark0817@kaist.ac.kr}
\author{Matthew Stoffregen}
\address{Department of Mathematics, Michigan State University, East Lansing, USA}
\email{stoffre1@msu.edu}

\makeatother

\providecommand{\corollaryname}{Corollary}
\providecommand{\definitionname}{Definition}
\providecommand{\remarkname}{Remark}
\providecommand{\theoremname}{Theorem}

\def\CF {\mathit{CF}}
\def\HF {\mathit{HF}}

\newcommand \CFm {\CF^-}
\newcommand \HFm {\HF^-}

\newcommand \Arf {\mathrm{Arf}}

\def\CFK{\mathit{CFK}}

\newcommand \gr{\mathrm{gr}}

\def\spinc {{\operatorname{spin^c}}}

\def\s{\mathfrak s}

\def\ff{\mathbb{F}}

\newcommand{\cU}{\mathcal{U}}
\newcommand{\cV}{\mathcal{V}}
\newcommand{\id}{\mathrm{id}}
\newcommand{\cC}{\mathcal{C}}
\newcommand{\ZZ}{\mathbb{Z}}
\newcommand{\Q}{\mathbb{Q}}
\newcommand{\N}{\mathbb{N}}

\newcommand{\sw}{\mathrm{sw}}
\newcommand{\exch}{\mathrm{exch}}

\begin{document}

\title{The $(2,1)$-cable of the figure-eight knot is not smoothly slice}

\begin{abstract}
We prove that the $(2,1)$-cable of the figure-eight knot is not smoothly slice by showing that its branched double cover bounds no equivariant homology ball. 
\end{abstract}
\thanks{ID was supported by NSF grant DMS-1902746. SK was supported by the Institute for Basic Science (IBS-R003-D1). AM was supported by the Max-Planck-Institut f\"ur Mathematik. JP was supported by Samsung Science and Technology Foundation (SSTF-BA2102-02) and the POSCO TJ Park Science Fellowship. MS was supported by NSF grant DMS-1952755.}
\maketitle

\section{Introduction}\label{sec:1}
In this paper, we prove that the $(2,1)$-cable of the figure-eight knot is not smoothly slice, answering a forty-year-old question posed by Kawauchi \cite{kawauchi1, kawauchi2}. This problem has attracted considerable attention (see for example \cite{banff,aim}) due to its connection to the slice-ribbon conjecture of Fox~\cite[Problem 25]{Fox:1962-1}. Explicitly, in \cite[Example 2]{miyazaki} Miyazaki showed that if $K$ is a fibered, negative amphicheiral knot with irreducible Alexander polynomial, then the $(2n,1)$-cable of $K$ is not (homotopy) ribbon for any $n \neq 0$. On the other hand, these knots are known to be strongly rationally slice (and thus algebraically slice) \cite{kawauchi1, Cha:2007-1, kw}.\footnote{See \cite[Definition 2]{kw} for a definition and discussion of strong rational sliceness.} While such cables are generally believed not to be slice, the fact that no argument has appeared in the literature has left open the possibility that these generate counterexamples to the slice-ribbon conjecture.

Prior to the current article, there has been no proof of nonsliceness for any of these knots, even in the simplest case of the $(2,1)$-cable of the figure-eight knot. For all of Miyazaki's examples, the usual concordance invariants defined using knot Floer homology \cite{OS-knots, rasmussen-knot-floer} vanish, as these are rational concordance invariants (see \cite{hom-survey} for a survey). One can additionally prove that for the $(2,1)$-cable of the figure-eight knot, the concordance invariants coming from involutive knot Floer homology \cite{hm} vanish (see \cite[Remark 1.4]{hom2020linear}) and the $s$-invariant \cite{rasmussen2010khovanov} from Khovanov homology \cite{khovanov} is zero.\footnote{This follows (for example) from \cite[Theorem 1.8]{manolescu2019generalization} and the fact that the $(2,1)$-cable of the figure-eight knot bounds nullhomologous slice disks in both $\mathbb{CP}^2$ and $\smash{\mathbb{\overline{CP}}^2}$.}

One general strategy for obstructing the sliceness of a knot is to show that its branched double cover bounds no $\ZZ/2\ZZ$-homology ball. This idea is well-known; see for example \cite{Cochran-Gompf:1988-1,Endo:1995-1,Lisca:2007-1, mow, Greene-Jabuka:2011-1, Hedden-Livingston-Ruberman:2012-1, Hedden-Pinzon:2021-1}. The new ingredient we use in this paper is to remember the data of the branching involution, which will allow us to obtain a more refined obstruction by asking for the existence of an \textit{equivariant} homology ball. This approach was employed successfully in \cite{aks, dhm} using tools from Heegaard Floer homology \cite{OS1, OS2}. Building on recent results from \cite{dms, mallick-surgery}, we extend these methods to prove:

\begin{thm}
\label{thm:mainthm1}
Let $K = 4_1$ be the figure-eight knot. Then the $(2,1)$-cable $K_{2,1}$ is not smoothly slice. In fact, $K_{2,1}$ has infinite order in the smooth concordance group.
\end{thm}
\noindent
This shows that $K_{2,1}$ has slice genus one. Our proof in fact implies $K_{2,1}$ is not slice in any $\ZZ/2\ZZ$-homology ball; see Remark~\ref{rem:homologyconcordance}. 


More generally, if $K$ is a fibered, negative amphicheiral knot with irreducible Alexander polynomial, then $K_{p, q} \# - T_{p, q}$ is rationally slice for any choice of cabling parameter $(p, q)$. (This can be seen by cabling the rational concordance from $K$ to the unknot provided by \cite{kw}.) For such $K$, work of Miyazaki \cite[Theorem 8.6]{miyazaki} again implies that $K_{p, q} \# - T_{p, q}$ is not ribbon, except in the trivial case when $p$ is zero. It is thus natural to investigate the nonsliceness of knots of the form $K_{p, q} \# - T_{p, q}$. Note that setting $(p, q) = (2, 1)$ reduces to studying the $(2,1)$-cable of $K$. Here, we show that the argument of Theorem~\ref{thm:mainthm1} extends to the general case $(p, q) = (2, k)$ under certain restrictions on the knot Floer homology of $K$:


\begin{thm}
\label{thm:mainthm2}
Let $K$ be a Floer-thin knot with $\Arf(K) = 1$ and let $k \in \N$ be odd. Then $K_{2, k} \# - T_{2, k}$ is not smoothly slice, and in fact has infinite order in the smooth concordance group.
\end{thm}
\noindent
The definition of Floer-thin may be found in \cite{mo}; all alternating or quasi-alternating knots are Floer-thin. Theorem~\ref{thm:mainthm2} can be applied to the next few of Miyazaki's examples after $K = 4_1$: these are $K = 6_3, 8_{12}$, and $8_{17}$. The smallest fibered, negative amphicheiral knot with irreducible Alexander polynomial which is not subsumed by Theorem~\ref{thm:mainthm2} is $10_{17}$. 

\section*{Acknowledgements}
The authors would like to thank Jae Choon Cha, Matthew Hedden, Peter Feller, and Ian Zemke for helpful discussions. 

\section{General Overview}\label{sec:2}
The topological outline of our strategy is quite straightforward. If $K$ is a knot with slice disk $D$, then the branched double cover $\Sigma_2(K)$ bounds a $\ZZ/2\ZZ$-homology ball given by the branched double cover $\Sigma_2(D)$. One can thus obstruct the sliceness of $K$ by showing that $\Sigma_2(K)$ does not bound any $\ZZ/2\ZZ$-homology ball, essentially moving the question into the setting of homology cobordism. In recent years, a refinement of this strategy has emerged: in addition to considering the 3-manifold $\Sigma_2(K)$, we may also remember the action of the branching involution $\tau \colon \Sigma_2(K) \rightarrow \Sigma_2(K)$. We can then ask if $\Sigma_2(K)$ bounds an \textit{equivariant} $\ZZ/2\ZZ$-homology ball; that is, a homology ball $W$ over which $\tau$ extends as a smooth involution $\tau_W \colon W \rightarrow W$. If $K$ has slice disk $D$, then such an extension is obtained by taking $\tau_W$ to be the branching action over $D$. Thus, asking for the existence of an equivariant homology ball also provides an obstruction to the sliceness of $K$, which is (in principle) a strictly stronger refinement of the previous.\footnote{For example, it is easy to see that $\Sigma_2((4_1)_{2,1})$ indeed bounds a $\ZZ/2\ZZ$-homology ball; see Remark~\ref{rem:corks}. Here, we show that it bounds no equivariant ball.} This strategy was employed successfully in several examples by the second author (in joint work with Alfieri and Stipsicz \cite{aks}) and the first and third authors (in joint work with Hedden \cite{dhm}); see also \cite{bh, kmt}.

In both \cite{aks} and \cite{dhm}, the obstruction to the existence of such a homology ball was established through the use of Heegaard Floer homology. We recall the essential features of this formalism here. Let $Y$ be a 3-manifold equipped with a $\spinc$-structure $\s$. In \cite{OS1, OS2}, Ozsv\'ath and Szab\'o construct a free, finitely-generated chain complex $\CFm(Y,\s)$ over the polynomial ring $\ff[U]$ (where $\ff = \ZZ/2\ZZ$) whose homotopy type is a diffeomorphism invariant of the pair $(Y, \s)$. (For a discussion and proof of naturality, see \cite{jtz}.) As explained in \cite{OS-triangles}, $\CFm(Y, \s)$ has a relative $\ZZ$-grading with $\gr(U) = -2$; if $\s$ is torsion then this can be lifted to an absolute $\Q$-grading. We briefly review the important aspects of Heegaard Floer homology which will be relevant for the present article. 

Firstly, we have the $\mathrm{TQFT}$ structure of Heegaard Floer homology \cite{OS-triangles}. Let $W$ be a cobordism from $Y_1$ to $Y_2$ and let $\s$ be a $\spinc$-structure on $W$ restricting to $\s_i = \s|_{Y_i}$. Then Ozsv\'ath and Szab\'o construct a Heegaard Floer cobordism map:
\[
F_{W, \s} \colon \CFm(Y_1, \s_1) \rightarrow \CFm(Y_2, \s_2).
\]
Secondly, by work of Juh\'asz, Thurston, and Zemke, the mapping class group $\mathrm{MCG}(Y)$ acts naturally on $\CFm(Y)$ \cite{jtz}. More precisely, let $\tau$ be a self-diffeomorphism of $Y$, and let $\s$ be a $\spinc$-structure on $Y$ which is invariant under pullback by $\tau$. Then $\tau$ induces a grading-preserving, $\ff[U]$-equivariant map
\[
\tau \colon \CFm(Y, \s) \rightarrow \CFm(Y, \s),
\]
which by abuse of notation we also denote by $\tau$. There is a slight subtlety in that generally speaking, only the \textit{based} mapping class group acts naturally on $\CFm(Y, \s)$. However, if $Y$ is a rational homology sphere (as will be the case in our setting), it follows from work of Zemke \cite{zemke-graph} that this subtlety can be disregarded; see \cite[Lemma 4.1]{dhm}. Finally, we also make essential use of the involutive Heegaard Floer package of Hendricks and Manolescu \cite{hm}. This is an enhancement of the usual Heegaard Floer formalism which utilizes the conjugation symmetry present in a Heegaard diagram. If $\s$ is a self-conjugate $\spinc$-structure on $Y$, then Hendricks and Manolescu define a grading-preserving, $\ff[U]$-equivariant homotopy involution
\[
\iota \colon \CFm(Y, \s) \rightarrow \CFm(Y, \s).
\]
The homotopy type of the pair $(\CFm(Y, \s), \iota)$ is a diffeomorphism invariant of $(Y, \s)$; $\iota$ should be thought of as an additional internal symmetry of $\CFm(Y, \s)$. The importance of considering $\iota$ in addition to $\tau$ was observed in \cite{dhm}; in the current article, the usage of $\iota$ will likewise be essential.

The crucial point is the following:

\begin{thm}\cite[Proof of Theorem 1.2]{OS-triangles, dhm}\label{thm:local}
Let $W$ be a rational homology cobordism between rational homology spheres $Y_1$ and $Y_2$. Let $\s$ be a $\spinc$-structure on $W$ restricting to $\s_i = \s|_{Y_i}$. Then the Heegaard Floer cobordism map 
\[
F_{W, \s} \colon \CFm(Y_1, \s_1) \rightarrow \CFm(Y_2, \s_2)
\]
satisfies the following properties:
\begin{enumerate}
    \item The Heegaard Floer grading shift of $F_{W, \s}$ is zero.
    \item The induced map on the localized homology 
    \[  
    H_*(U^{-1}\CFm{(Y_1, \s_1)}) \xrightarrow{\cong} H_*(U^{-1}\CFm{(Y_2, \s_2)})
    \]
    is an isomorphism. 
\end{enumerate}
Suppose moreover that $\s$ is self-conjugate. Then:
\begin{enumerate}
    \item[(3)] Up to homotopy, $F_{W,\s}$ intertwines the $\iota$-action on $\CFm(Y_1, \s_1)$ with the $\iota$-action on $\CFm(Y_2, \s_2)$:
    \[
    F_{W,\s} \circ \iota_1 \simeq \iota_2 \circ F_{W, \s}.
    \]
\end{enumerate}
Finally, let $\tau$ be a self-diffeomorphism of $W$ restricting to $\tau_i = \tau|_{Y_i}$. Suppose $\s$ is invariant under $\tau$. Then:
\begin{enumerate}
    \item[(4)] Up to homotopy, $F_{W, \s}$ intertwines the action of $\tau_1$ on $\CFm(Y_1, \s_1)$ with the action of $\tau_2$ on $\CFm(Y_2, \s_2)$:
    \[
    F_{W,\s} \circ \tau_1 \simeq \tau_2 \circ F_{W, \s}.
    \]
\end{enumerate}
\end{thm}
\noindent
The first two parts of Theorem~\ref{thm:local} are standard results in the literature, while the third is work of Hendricks and Manolescu \cite{hm}. The only subtlety in establishing the remaining claim is to analyze the choice of path needed to define $F_{W, \s}$ (see \cite{zemke-graph}), since this may not be invariant under $\tau$; see \cite[Section 6]{dhm}. This is done in the proof of \cite[Theorem 1.2]{dhm} (in fact, for a much wider class of cobordisms). We call a map satisfying the conditions of Theorem~\ref{thm:local} a \textit{local map}, or sometimes a 
\textit{$(\tau, \iota)$-local map} for emphasis. Given two complexes $C_1$ and $C_2$ as above with actions of $\iota$ and $\tau$, we say that $C_1$ and $C_2$ are \textit{locally equivalent} if there are local maps $f \colon C_1 \rightarrow C_2$ and $g \colon C_2 \rightarrow C_1$ in both directions. Occasionally, we will use this terminology but relax the grading shift condition when our meaning is clear.

To see the relevance of Theorem~\ref{thm:local}, let $Y$ be a $\ZZ/2\ZZ$-homology ball with involution $\tau$ and suppose there exists an equivariant $\ZZ/2\ZZ$-homology ball $W$ with boundary $Y$. Note that $W$ has a unique self-conjugate $\spinc$-structure $\s$ which is necessarily preserved by the extension $\tau_W$; write $\s|_Y = \s_0$. Puncturing $W$ gives a $(\tau, \iota)$-local equivalence between the trivial complex $\CFm(S^3) = \ff[U]$ and $\CFm(Y, \s_0)$, with the actions of $\iota$ and $\tau$ on the former being the identity. Here and throughout, we normalize gradings so that the uppermost generator of $\CFm(S^3)$ lies in Maslov grading zero; this differs slightly from the usual convention in which $\CFm(S^3)$ starts in Maslov grading $-2$.

Our strategy will hence be to compute the actions of $\iota$ and $\tau$ on the Heegaard Floer complex of $\Sigma_2(K_{2,1})$ and then algebraically rule out the existence of a local map from the trivial complex. Techniques for establishing the nonexistence of maps as in Theorem~\ref{thm:local} have been developed by many authors and have led to the formalism of the \textit{local equivalence group}; see especially the work of Hendricks, Manolescu, and Zemke \cite{hmz}. Although it will be helpful for the reader to have a broad familiarity with these ideas, for the main claim of Theorem~\ref{thm:mainthm1} this will not be needed. See \cite[Section 3]{aks} or \cite[Section 3]{dhm} for a more comprehensive discussion of local equivalence in this setting.

The main crux of this paper will thus be the computation of $\iota$ and $\tau$. In general, however, determining the induced action of $\tau$ is difficult. Here, we make essential use of the following topological observation: for any knot $K$, the Montesinos trick \cite{montesinos} shows that 
\[
\Sigma_2(K_{2,1}) \cong S^3_{+1}(K \# K^r).
\]
Moreover, this homeomorphism identifies the branching action on $\Sigma_2(K_{2,1})$ with the involution on $S^3_{+1}(K \# K^r)$ induced from the obvious strong inversion on $K \# K^r$ interchanging the two factors. (See \Cref{fig:obviousinv}.) We refer to this strong inversion as the \textit{swapping involution}; see \cite[Section 1.1.12]{saveliev} for further details.

\begin{figure}[hbt]
    \centering
    \includegraphics[width=0.48\textwidth]{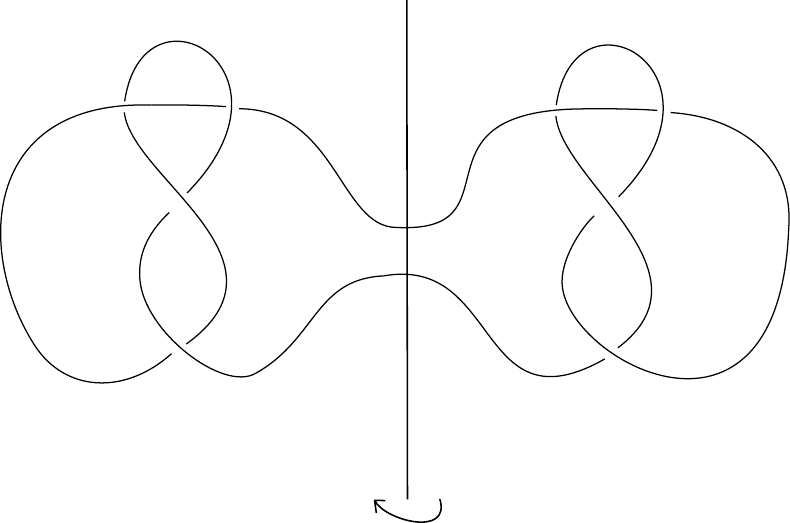}
    \caption{The swapping involution on $K\# K^r$. This is constructed as follows: fix an axis of rotation far away from $K$, and rotate $K$ about this axis by $180^\circ$ degrees. Reversing orientation on the resulting image and summing along the obvious band gives $K \# K^r$.}
\label{fig:obviousinv}
\end{figure}

In recent work \cite{dms}, the first, third, and fifth authors carried out a program for studying involutions on knots using the knot Floer homology package. (For background on knot Floer homology, see \cite{OS-knots, rasmussen-knot-floer}.) Explicitly, in \cite[Section 3]{dms} it is shown that a strong inversion on a knot induces an automorphism of its knot Floer complex (see also \cite[Section 2.2]{mallick-surgery}). Building on work of Juh\'asz and Zemke \cite[Theorem 8.1]{jz}, in \cite[Theorem 4.3]{dms} the authors computed the induced action for such examples $K \# K^r$ as in Figure~\ref{fig:obviousinv}. Combining this with a large surgery formula due to the third author \cite[Theorem 1.1]{mallick-surgery}, this will allow us to compute the induced action of $\tau$ on $\smash{S^3_n(K \# K^r)}$ for $n$ large. We then use an auxiliary argument involving an equivariant, negative definite cobordism from $n$-surgery to $(+1)$-surgery to constrain the induced action of $\tau$ on $\smash{S^3_{+1}(K \# K^r)}$. By the above topological observation, this is the same as the branching action on $\Sigma_2(K_{2,1})$.

\begin{rem}\label{rem:corks}
The discussion of this section can be summarized succinctly as follows. Taking double branched covers yields a homomorphism from the smooth concordance group to the $(\ZZ/2\ZZ)$-homology cobordism group $\smash{\Theta^3_{\ZZ/2\ZZ}}$. In \cite[Section 2]{dhm}, a refinement of $\smash{\Theta^3_{\ZZ/2\ZZ}}$ called the \textit{$(\ZZ/2\ZZ)$-homology bordism group of involutions} was defined. Denoted by $\smash{\Theta^\tau_{\ZZ/2\ZZ}}$, this additionally takes into account an involution on each homology sphere. Remembering the branching action gives a homomorphism from $\mathcal{C}$ to $\smash{\Theta^\tau_{\ZZ/2\ZZ}}$, through which the previous homomorphism factors:
\[
\Sigma_2 \colon \mathcal{C} \xrightarrow{2\text{x cover}} \Theta^\tau_{\ZZ/2\ZZ} \xrightarrow{\text{forgetful map}} \Theta^3_{\ZZ/2\ZZ}.
\]
The approach of the present article is to show that $K_{2,1}$ is not slice by showing that its image is nonzero in $\smash{\Theta^\tau_{\ZZ/2\ZZ}}$. Note that (for example) $(4_1)_{2,1}$ \textit{does} have trivial image in $\smash{\Theta^3_{\ZZ/2\ZZ}}$: since $4_1 \# 4_1^r$ is slice, $S^3_{+1}(4_1 \# 4_1^r)$ bounds a homology ball (in fact, a contractible manifold). However, $4_1 \# 4_1^r$ is not \textit{equivariantly} slice; while $4_1^r$ is isotopic to $-4_1$, this isotopy does not preserve the swapping involution. Readers familiar with the work of Lin, Ruberman, and Saveliev \cite{lrs} will recognize that we show $S^3_{+1}(4_1 \# 4_1^r)$ (with the branching involution) is a $(\ZZ/2\ZZ)$-\textit{strong cork}. See \cite[Section 1.2.2]{lrs} or \cite[Definition 1.8]{dhm} for further discussion.
\end{rem}

\section{Analysis of large surgery}\label{sec:3}

We now describe how to compute the action of the swapping involution on $\CFK(K \# K^r)$. Throughout, we use Zemke's conventions for knot Floer homology \cite{zemke-link, zemke-gradings}. This assigns to a (doubly-based, oriented) knot $K$ a bigraded complex $\CFK(K)$ over the polynomial ring $\ff[\cU, \cV]$. Denote the two gradings by $\gr = (\gr_U, \gr_V)$; we have $\gr(\cU) = (-2,0)$, $\gr(\cV) = (0, -2)$, and $\gr(\partial) = (-1, -1)$. See \cite[Section 1.5]{zemke-gradings} for further discussion.

Building on the naturality results of \cite{zemke-link}, in \cite[Section 3]{dms} it is shown that a strong inversion $\tau$ on $K$ induces a well-defined homotopy involution $\tau_K$ on $\CFK(K)$. This is a skew-graded, skew-equivariant automorphism of $\CFK(K)$ in the sense that $\tau_K$ interchanges $\gr_U$ and $\gr_V$ and satisfies $\tau_K(\cU^i \cV^j x) = \cV^i \cU^j \tau_K(x)$. It will also be important for us to consider the involutive knot Floer map defined by Hendricks and Manolescu in \cite{hm}. This is similarly a skew-graded, skew-equivariant automorphism of $\CFK(K)$ which we denote by $\iota_K$. See \cite[Section 6]{hm} for further discussion.

\subsection{The swapping involution}
Fix any knot $K$. By a slight abuse of notation, throughout this subsection we write $\tau$ for the $180^\circ$-degree rotation about an axis far away from $K$. We form the connected sum $K \# \tau K^r$ by reversing orientation on $\tau K$ and summing along an obvious band as in Figure~\ref{fig:obviousinv}. While $\tau K$ is of course isotopic to $K$, in order to discuss the action of the swapping involution it will be helpful for us to clearly distinguish between the two. We thus follow the notation of \cite[Theorem 4.3]{dms} and temporarily write $K \# \tau K^r$, rather than $K \# K^r$. Denote the swapping involution on $K \# \tau K^r$ by $\tau_\sw$.

The following tensor product formula allows us to simultaneously calculate the action of $\tau_\sw$ and the involutive knot Floer map on $\CFK(K \# \tau K^r)$:

\begin{thm}{{\cite[Theorem 4.3]{dms}}}\label{thm:swapping}
The triples
\[
(\CFK(K \# \tau K^r), \tau_\sw, \iota_{K \#  \tau K^r}) \quad \text{and} \quad (\CFK(K) \otimes \CFK(\tau K^r), \tau_\otimes, \iota_\otimes)
\]
are homotopy equivalent, where
\[
\tau_\otimes = (\id \otimes \id + \Psi \otimes \Phi) \circ \tau_\exch 
\]
and
\[
\iota_\otimes = \varsigma_\otimes \circ (\id \otimes \id + \Psi \otimes \Phi) \circ (\iota_K \otimes \iota_{\tau K^r}).
\]
\end{thm}
\noindent
For convenience of the reader, we briefly describe the terms appearing in Theorem~\ref{thm:swapping}. The most important map to explain is $\tau_\exch$, which should be thought of as the induced action of $\tau$ on the Floer homology of $K \sqcup \tau K^r$. Note that $\tau$ interchanges the two components of $K \sqcup \tau K^r$ and reverses orientation on both; this parallels the fact that a strong inversion on a knot is orientation-reversing. We first give an actual definition of $\tau_\exch$ for the reader more familiar with knot Floer homology, followed by an explanation of how to compute $\tau_\exch$ in practice.

Let $w$ and $z$ be an ordered pair of basepoints on $K$. Fixing Heegaard data $\mathcal{H}$ compatible with $w$, $z$, and the orientation on $K$ defines the knot Floer complex $\CFK(K, w, z)$. There are three other closely-related sets of Heegaard data which we will need to consider:
\begin{enumerate}
    \item Keeping the same Heegaard splitting (including the designation of the $\alpha$- and $\beta$-curves and the orientation of the Heegaard surface) but interchanging the roles of $w$ and $z$ gives a set of Heegaard data for the reverse of $K$. Denote the corresponding knot Floer complex by $\CFK(K^r, z, w)$.
    \item The pushforward of $\mathcal{H}$ under $\tau$ forms a set of Heegaard data for $(\tau K, \tau w, \tau z)$. Denote the corresponding knot Floer complex by $\CFK(\tau K, \tau w, \tau z)$. 
    \item Taking the pushforward of $\mathcal{H}$ under $\tau$ and then interchanging the roles of $\tau w$ and $\tau z$ gives a set of Heegaard data for $(\tau K^r, \tau z, \tau w)$. Denote the corresponding knot Floer complex by $\CFK(\tau K^r, \tau z, \tau w)$.
\end{enumerate}
These complexes are all isomorphic to each other in different ways. Firstly, note that $\tau$ induces a tautological isomorphism from $\CFK(K, w, z)$ to $\CFK(\tau K, \tau w, \tau z)$, and likewise an isomorphism from $\CFK(\tau K^r, \tau z, \tau w)$ to $\CFK(K^r, z, w)$. Using this, define a map
\[
t \colon \CFK(K,w,z) \otimes \CFK( \tau K^r,\tau z, \tau w)  \rightarrow \CFK(K^{r},z,w) \otimes \CFK(\tau K, \tau w, \tau z)
\]
which takes the first factor on the left isomorphically onto the second factor on the right, and the second factor on the left isomorphically onto the first factor on the right.

Secondly, observe that we have a tautological map $\sw \colon \CFK(K^r, z, w) \rightarrow \CFK(K, w, z)$ given by mapping intersection-points to intersection-points. Unlike the previous isomorphisms, this is skew-graded and skew-equivariant. We call $\sw$ the switch map; see \cite[Definition 3.1]{dms} for further discussion. 
Similarly, there is a switch map from $\CFK(\tau K, \tau w, \tau z)$ to $\CFK(\tau K^r, \tau z, \tau w)$. Define
\[
\sw \otimes \sw \colon \CFK(K^{r},z,w) \otimes \CFK(\tau K,\tau w,\tau z) \rightarrow \CFK(K,w,z) \otimes \CFK(\tau K^{r},\tau z,\tau w)
\]
by taking the switch map in both factors. We then define $\tau_\exch$ to be the composition
\[
\tau_\exch = (\sw \otimes \sw) \circ t. 
\]
Note that $\tau_\exch$ is skew-graded and skew-equivariant. Roughly speaking, $t$ exchanges the complexes of $K$ and $\tau K^r$ using the tautological pushforward under $\tau$, but we must apply $\sw \otimes \sw$ to account for the fact that $\tau$ is orientation-reversing on each component.

In practice, computing the action of $\tau_\exch$ is quite straightforward. Start with a knot Floer complex $\CFK(K) = \CFK(K,w,z)$ for $K$, as schematically depicted on the left in Figure~\ref{fig:exchange}. Reflecting this complex over a diagonal line gives the complex $\CFK(K^r) = \CFK(K^r, z, w)$ for $K^r$. This is the same (via the tautological pushforward) as the complex $\CFK(\tau K^r) = \CFK(\tau K^r, \tau z, \tau w)$ for $\tau K^r$. Note that the switch map isomorphism from $\CFK(K)$ to $\CFK(K^r)$ given by reflection over the diagonal line is indeed a skew-equivariant, skew-graded chain map. The action of $\tau_\exch$ on $\CFK(K) \otimes \CFK(\tau K^r)$ is then given by the interchange-and-reflect operation, as in the right of Figure~\ref{fig:exchange}.

\begin{figure}[hbt!]
    \centering
    \includegraphics[width=0.97\textwidth]{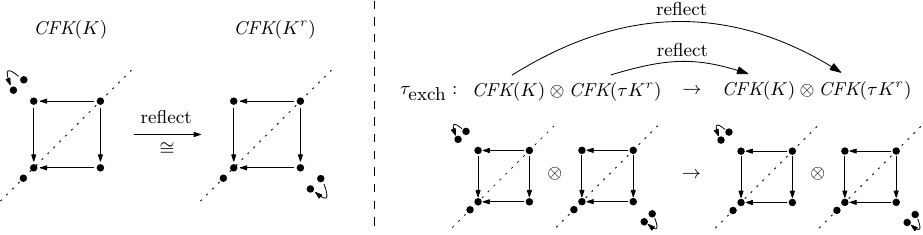}
    \caption{A schematic depiction of $\tau_{\exch}$. Here, we have deliberately chosen an asymmetric complex for $\CFK(K)$ to better illustrate the switch map from $\CFK(K)$ to $\CFK(K^r)$.}
\label{fig:exchange}
\end{figure}

The second subtlety in Theorem~\ref{thm:swapping} is the determination of the knot Floer involution $\iota_{\tau K^r}$ on $\CFK(\tau K^r) = \CFK(\tau K^r, \tau z, \tau w)$. In Figure~\ref{fig:exchange} we have identified $\CFK(\tau K^r, \tau z, \tau w)$ with the reflection of $\CFK(K, w, z)$. However, the map $\iota_{\tau K^r}$ should \textit{not} be identified with the reflection of the map $\iota_K$ on $\CFK(K, w, z)$. This is because the skew isomorphism
\[
\sw \colon \CFK(K, w, z) \rightarrow \CFK(K^r, z, w)
\]
does \textit{not} intertwine $\iota_K$ on the left with $\iota_{K^r}$ on the right. Instead, it follows from arguments as in \cite[Lemma 3.4]{dms} or \cite[Lemma 3.8]{dms} that $\sw$ intertwines $\iota_K$ on the left with $\varsigma_{K^r} \circ \iota_{K^r}$ on the right, where $\varsigma_{K^r}$ is the Sarkar map (see below). Indeed, for any oriented knot $K$, let $\rho_K$ be the pushforward isomorphism on the knot Floer complex induced by the half-Dehn twist along the orientation of $K$; denote the half-Dehn twist in the opposite direction by $\bar{\rho}_K$. Recall from \cite[Section 6]{hm} that $\iota_K$ is defined by composing the conjugation map $\eta$ with $\rho_K$ (followed by a naturality map returning to the original Heegaard data). By \cite[Lemma 3.4]{dms},

\[
\iota_{K^r} \circ \sw = (\rho_{K^r} \circ \eta) \circ \sw = \sw \circ (\bar{\rho}_K \circ \eta) \simeq \sw \circ (\bar{\rho}_K^2) \circ (\rho_K \circ \eta) = \sw \circ (\varsigma_K^{-1} \circ \iota_K).
\]
Roughly speaking, the point is that $\iota_K$ is defined using a half-Dehn twist following the orientation of $K$. However, from the point of view of $K^r$, this half-Dehn twist goes \textit{against} the orientation of the knot in question; see \cite[Section 3.1]{dms} for further discussion. Using the fact that $\varsigma_K$ is its own (homotopy) inverse, this shows 
\begin{equation}\label{equ: iota_ori_relation}
\iota_{K^r} \simeq \sw \circ (\varsigma_K \circ \iota_K) \circ \sw.
\end{equation}
Finally, we briefly remind the reader of the meaning of the maps $\Psi$, $\Phi$, and $\varsigma_K$. The first two of these were introduced by Zemke in \cite[Section 3]{zemke-stabilization} and are the formal derivatives of $\partial$ with respect to $\cU$ and $\cV$:
\[
\Phi = \left( \dfrac{d}{d\cU} \partial \right) \quad \text{and} \quad \Psi = \left( \dfrac{d}{d\cV} \partial \right).
\]
The map $\varsigma_K$ is the Sarkar basepoint-moving map \cite{sarkar} on the oriented knot $K$. In \cite[Theorem B]{zemke-stabilization}, it is shown that
\[
\varsigma_K \simeq \id + \Phi \circ \Psi.
\]
In the statement of Theorem~\ref{thm:swapping}, $\Psi$ acts on $\CFK(K)$ while $\Phi$ acts on $\CFK(\tau K^r)$. The map $\varsigma_{\otimes}$ is the Sarkar map on the entire tensor product complex $\CFK(K) \otimes \CFK(\tau K^r)$; this means that (following Zemke \cite{zemke2019connected}) we define $\varsigma_{\otimes}$ to be
\[
\varsigma_{\otimes} := \id_{\otimes} + \Phi_{\otimes} \circ \Psi_{\otimes}.
\]

We are now finally ready to put everything together. Let $K = 4_1$. As shown in \cite[Section 8.2]{hm}, the knot Floer complex $\CFK(K)$ for $K$ (together with the action of $\iota_K$) is homotopy equivalent to the complex displayed on the left in Figure~\ref{fig:factors}. Following the above discussion, we reflect this complex across a diagonal line to obtain $\CFK(\tau K^r)$. The action of $\iota_{\tau K^r}$ is displayed on the right in Figure~\ref{fig:factors}, which is computed using (\ref{equ: iota_ori_relation}).

\begin{figure}[hbt!]
    \centering
    \includegraphics[width=0.65\textwidth]{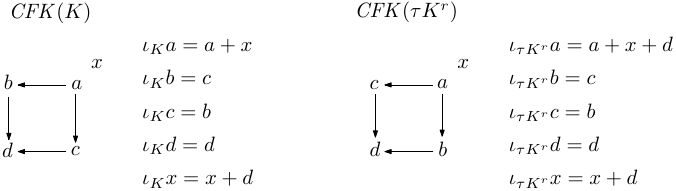}
    \caption{The knot Floer complex $\CFK(K)$ (left) and $\CFK(\tau K^r)$ (right). Horizontal and vertical arrows denote $\cU$ and $\cV$ terms in the differential. We label the generators by $\{x, a, b, c, d\}$ in both cases, so that $\tau_{\exch}(p \otimes q) = q \otimes p$. On the left, the gradings of these are as follows: $\gr(x) = \gr(a) = \gr(d) = (0,0)$, $\gr(b) = (1, -1)$, and $\gr(c) = (-1, 1)$. The gradings on the right are the same, except $\gr(b) = (-1, 1)$ and $\gr(c) = (1, -1)$.}
\label{fig:factors}
\end{figure}

We may then take the tensor product of $\CFK(K)$ and $\CFK(\tau K^r)$ and perform the algebra of Theorem~\ref{thm:swapping} to compute $\tau_\sw$ and $\iota_{K \# \tau K^r}$. In Figure~\ref{fig:chaincomplex}, we have displayed an especially convenient basis for $\CFK(K) \otimes \CFK(\tau K^r)$, which will be useful presently.

However, since we will not actually need to know the actions of $\tau_\sw$ and $\iota_{K \#  \tau K^r}$ on the entire complex, we suppress the full computation. Instead, it will suffice for us to discuss the action on the large surgery subcomplex, which we do below.

\begin{figure}[hbt!]
    \centering
    \includegraphics[width=0.5\textwidth]{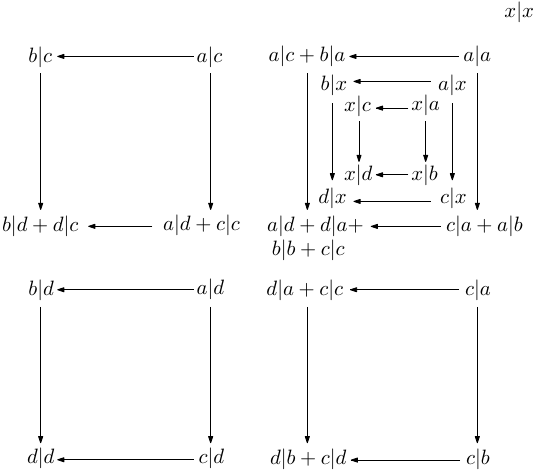}
    \caption{The knot Floer complex of $K\# K^r$. Horizontal and vertical arrows denote $\cU$ and $\cV$ terms in the differential.}
\label{fig:chaincomplex}
\end{figure}

\begin{rem}
The reader familiar with involutive knot Floer homology may note that the connected sum involution $\iota_\otimes$ in Theorem~\ref{thm:swapping} appears to be different from the connected sum formula of \cite[Theorem 1.1]{zemke2019connected}. This is not actually material: it follows from \cite[Lemma 2.21]{dms} that defining
\[
\tau_\otimes' = \varsigma_\otimes \circ (\id \otimes \id + \Psi \otimes \Phi) \circ \tau_\exch
\]
and
\[
\iota_\otimes' = (\id \otimes \id + \Psi \otimes \Phi) \circ (\iota_K \otimes \iota_{\tau K^r}).
\]
results in a homotopy equivalent triple, bringing us into alignment with \cite[Theorem 1.1]{zemke2019connected}. Likewise, the reader may wonder why the formula for $\tau_\otimes$ is not simply $\tau_\exch$. As in \cite[Theorem 1.1]{zemke2019connected}, this is due to a subtlety involving the definition of $\tau_K$ and the choice of decoration for the usual pair-of-pants cobordism from $K \sqcup \tau K^r$ to $K \# \tau K^r$; see \cite[Section 4.2]{dms}. In fact, one can show that taking $\tau_\otimes = \tau_\exch$ does not satisfy the commutation relation with $\iota_\otimes$ required by \cite[Theorem 1.7]{dms}, unlike the expression for $\tau_\otimes$ in Theorem~\ref{thm:swapping}. 
\end{rem}

\subsection{The large surgery formula}
For any knot $K$, let $A_0(K)$ denote the subset of $\CFK(K)$ lying in Alexander grading $A = (\gr_U - \gr_V)/2 = 0$. We think of this as a singly-graded chain complex over $\ff[U]$, where $U = \cU \cV$ and the Maslov grading is given by $\gr_U = \gr_V$. It is straightforward to check that $\iota_K$ and $\tau_K$ both preserve $A_0(K)$. We may thus view both of these as $\ff[U]$-equivariant, grading-preserving chain automorphisms on $A_0(K)$, which (by abuse of notation) we continue to denote by $\iota_K$ and $\tau_K$. 

The large surgery formula of Ozsv{\'a}th and Szab{\'o} \cite{OS-knots} (see also \cite{rasmussen-knot-floer}) states that if $n \geq g_3(K)$, then we have a homotopy equivalence of $\ff[U]$-complexes
\[
\Gamma_{n, 0} \colon \CFm(S^3_n(K), \s_0) \xrightarrow{\simeq} A_0(K).
\]
Here, $\s_0$ denotes the (self-conjuate) $\spinc$-structure on $S^3_n(K)$ which corresponds to the element $0\in \mathbb{Z}/n\mathbb{Z}$ under the identification defined in \cite[Lemma 2.2]{ozsvath2008knot}. This homotopy equivalence is relatively graded with grading shift $-(n-1)/4$.\footnote{That is, an element of $\CFm(S^3_n(K), \s_0)$ with grading $(n-1)/4$ has Maslov grading zero in $A_0(K)$.} Moreover, by \cite[Theorem 6.8]{hm} and \cite[Theorem 1.1]{mallick-surgery}, $\Gamma_{n, 0}$ intertwines $\iota$  with $\iota_K$ and $\tau$ with $\tau_K$ (up to homotopy), so that we get a homotopy equivalence of triples
\[
\Gamma_{n, 0} \colon (\CFm(S^3_n(K), \s_0), \tau, \iota) \xrightarrow{\simeq} (A_0(K), \tau_K, \iota_K).
\]
See the proof of \cite[Corollary 1.3]{mallick-surgery}.

We now return to the case of $K = 4_1$. In light of the large surgery formula, we restrict our attention to the large surgery subcomplex $A_0(K \# K^r)$. In fact, it turns out that for our purposes it suffices to understand the actions of $\tau_{\sw}$ and $\iota_{K \# K^r}$ on the homology of this complex. From Figure~\ref{fig:chaincomplex}, it is straightforward to verify that $H_*(A_0(K \# K^r))$ is generated by the classes
\[
[x|x], \quad [x|d], \quad [d|x], \quad [a|d+d|a+b|b+c|c], \quad \text{and} \quad [d|d],
\]
all lying in Maslov grading zero. (We temporarily use $|$ in place of $\otimes$ for brevity.) The first of these is $U$-nontorsion, while the others have $U$-torsion order one. In Table~\ref{table:computation}, we have computed the actions of $\iota$ and $\tau$ on these using Theorem~\ref{thm:swapping}. To illustrate the calculation, we do the most complicated case of $a|d + d|a + b|b + c|c$ and leave the other entries to the reader. We have:
\begin{align*}
    (\iota_K | \iota_{\tau K^r})(a|d + d|a + b|b + c|c) &= (a + x)|d + d|(a + x + d) + c|c + b|b \\
    &= a|d + d|a + b|b + c|c + x|d + d|x + d|d.
\end{align*}
Among these terms, the action of $\Psi|\Phi$ vanishes outside of $(\Psi | \Phi)(b|b) = d|d$. The action of $\varsigma_\otimes$ is then easily checked to be the identity; hence $\iota_\otimes = \varsigma_\otimes \circ (\id | \id + \Psi | \Phi) \circ (\iota_K | \iota_{\tau K^r})$ is as displayed on the fourth row of Table~\ref{table:computation}. Similarly,
\begin{align*}
    \tau_\exch(a|d + d|a + b|b + c|c) &= d|a + a|d + b|b + c|c.
\end{align*}
The same computation of $\Psi|\Phi$ then gives the action of $\tau_\otimes = (\id | \id + \Phi | \Psi) \circ \tau_\exch$ in the fourth row of Table~\ref{table:computation}.

\begin{table}
\begin{tabular}{|c||c|c|c|c|}
 \hline
 Cycles & Homology class & Image under $\iota$ & Image under $\tau$ \\
 \hline
 $x|x$ & free & $x|x+x|d+d|x+d|d$ & $x|x$ \\
 $x|d$ & $U$-torsion & $x|d+d|d$ & $d|x$ \\
 $d|x$ & $U$-torsion & $d|x+d|d$ & $x|d$ \\
 $a|d+d|a+b|b+c|c$ & $U$-torsion & $a|d+d|a+b|b+c|c+x|d+d|x$ & $a|d+d|a+b|b+c|c+d|d$ \\
 $d|d$ & $U$-torsion & $d|d$ & $d|d$ \\
 \hline
\end{tabular}
\captionsetup{justification=centering}
\caption{Actions of $\iota$ and $\tau$ on $H_*(A_0(4_1 \#
4_1^r))$.}\label{table:computation}
\end{table}

The crucial observation is now the following:
\begin{lem}\label{lem:towers}
There is no $U$-nontorsion class in $H_*(A_0(4_1 \# 4_1^r))$ which has Maslov grading zero and is invariant under both $\iota$ and $\tau$. 
\end{lem}
\begin{proof}
With respect to the ordered basis $\{[x|x],[x|d],[d|x],[a|d+d|a+b|b+c|c],[d|d]\}$, the action of $\id +\iota$ is given by the matrix
\[
\id +\iota = \begin{pmatrix}
0 & 0 & 0 & 0 & 0 \\ 1 & 0 & 0 & 1 & 0 \\ 1 & 0 & 0 & 1 & 0 \\ 0 & 0 & 0 & 0 & 0 \\ 1 & 1 & 1 & 0 & 0
\end{pmatrix}.
\]
Hence we see that the $\iota$-invariant subspace of $H_*(A_0(4_1 \# 4_1^r))$ is generated by 
\[
[x|x]+[x|d]+[a|d+d|a+b|b+c|c], \quad [x|d]+[d|x], \quad \text{and} \quad [d|d].
\]
Among these, the homology classes $[x|d]+[d|x]$ and $[d|d]$ are $\tau$-invariant, while we have
\[
(\id+\tau)([x|x]+[x|d]+[a|d+d|a+b|b+c|c]) = [x|d] + [d|x] + [d|d].
\]
Hence the $(\tau, \iota)$-invariant subspace of $H_{\ast}(A_0(4_1 \# 4_1^r))$ is generated by $[x|d] + [d|x]$ and $[d|d]$, both of which are annihilated by $U$.
\end{proof}

\section{Reduction to small surgery}\label{sec:4}
We now present the construction which will allow us to pass from large surgery to small surgery. Let $K$ be any strongly invertible knot. For any $n > 0$, define a cobordism $W_{1, n}$ whose incoming end is $S^3_{+1}(K)$ by attaching 2-handles along $(-1)$-framed meridians of $K$, as indicated in Figure~\ref{fig:oddcobordism}. We denote these meridians by $x_1, \ldots, x_{n-1}$. Clearly, blowing down all of the $x_i$ shows that the outgoing end of the cobordism is diffeomorphic to $S^3_{n}(K)$. We can always arrange $W_{1, n}$ to be equivariant: if $n$ is odd, then all of the 2-handles may be divided into symmetric pairs, whereas if $n$ is even, then there is a single 2-handle which intersects the axis of symmetry.

\begin{figure}[hbt!]
    \centering
    \includegraphics[width=0.23\textwidth]{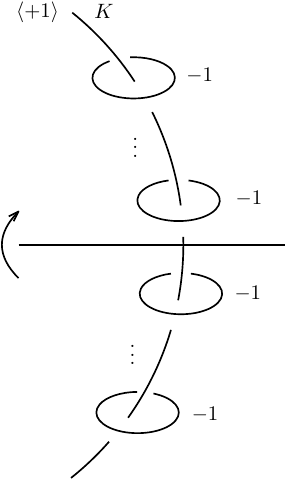}
    \caption{The cobordism $W_{1,n}$, formed by attaching $(-1)$-framed $2$-handles along $n-1$ meridians of $K$.}
\label{fig:oddcobordism}
\end{figure}

\begin{lem}\label{lem:intersectionform}
The cobordism $W_{1, n}$ is spin and negative definite.
\end{lem}
\begin{proof}
We calculate the intersection form of the cobordism. Orient the $x_i$ so that $x_i \cdot K = 1$. We isotope the attaching curves of the $2$-handles by sliding each $x_i$ over $K$ to obtain curves $x_i' = x_i - K$. By this, we mean that we isotope the $x_i$ by passing them through the surgery solid torus in $S^3_{+1}(K)$; alternatively, consider $S^3_{+1}(K)$ as the boundary of $(+1)$-framed $2$-handle attachment along $K$ and perform a $2$-handle slide. Since $x_i' \cdot K = 0$, the intersection form of $W_{1, n}$ may be computed by calculating the pairings $x_i' \cdot x_j'$. We have $x_i' \cdot x_i' = -2$ for all $i$, while $x_i' \cdot x_j' = -1$ for all $i \neq j$. Since both of these terms are negative, the cobordism is evidently negative definite; since the intersection form is even, the cobordism is spin.
\end{proof}

The existence of $W_{1, n}$ gives us partial information regarding the $\tau$- and $\iota$-actions on $S^3_{+1}(K)$ in terms of those on $S^3_n(K)$. More precisely, we have the following:

\begin{lem}\label{lem:cobordism}
Let $K$ be any strongly invertible knot. If $n$ is odd, then there is an $(\tau,\iota)$-local map
\[
f_{1, n} \colon \CFm(S^3_{+1}(K)) \rightarrow \CFm(S^3_n(K), \s_0)
\]
with absolute grading shift $(n-1)/4$. Here, $\s_0$ is the unique self-conjugate $\spinc$-structure on $S^3_n(K)$.
\end{lem}
\begin{proof}
By Lemma~\ref{lem:intersectionform}, we have that $W_{1, n}$ is spin and negative definite. Let $\s$ be the unique spin structure on $W_{1, n}$. (Note that the homology of $W_{1, n}$ has no torsion; hence it has a unique spin structure which moreover corresponds to a unique self-conjugate $\spinc$-structure.) Since $n$ is odd, we necessarily have $\s|_{S^3_n(K)} = \s_0$. As discussed above, $W_{1, n}$ may be taken to be equivariant, and by uniqueness, $\s$ is fixed by the involution on $W_{1, n}$. Hence the cobordism map
\[
F_{W_{1, n}, \s} \colon \CFm(S^3_{+1}(K)) \rightarrow \CFm(S^3_n(K), \s_0)
\]
homotopy commutes with $\tau$ (and $\iota$), as in \cite[Section 6]{dhm}. It remains to compute the resulting grading shift, which is easily seen to be 
\[
\dfrac{2 \sigma(W_{1,n}) + 3 \chi(W_{1,n})}{4} = \dfrac{n-1}{4}
\]
as desired.
\end{proof}

The application of Lemma~\ref{lem:cobordism} will become clear in the next section.

\section{Proof of Theorem \ref{thm:mainthm1}}\label{sec:5}


We now turn to the proof of the main theorem. First, we have the following obstruction to bounding an equivariant homology ball:
\begin{lem}\label{lem:easy-local-equivalence}
If $Y$ bounds an equivariant $\ZZ/2\ZZ$-homology ball, then $\HFm(Y, \s_0)$ admits a $U$-nontorsion class in Maslov grading zero which is invariant under both $\iota$ and $\tau$. Here, $\s_0$ is the unique self-conjugate $\spinc$-structure on $Y$.
\end{lem}
\begin{proof}
Let $W$ be an equivariant $\ZZ/2\ZZ$-homology ball with boundary $Y$. Let $\s$ be the unique self-conjugate $\spinc$-structure on $W$. Puncturing $W$ gives an equivariant $\ZZ/2\ZZ$-homology cobordism, and it follows from Theorem~\ref{thm:local} that the map associated to this cobordism
\[
F_{W, \s} \colon \CFm(S^3) \rightarrow \CFm(Y, \s_0)
\]
is local. The image $F_{W, \s}(1)$ gives the desired class.
\end{proof}

This immediately shows that $(4_1)_{2,1}$ is not slice:

\begin{proof}[Proof of \Cref{thm:mainthm1}]
Suppose that $(4_1)_{2,1}$ were slice. Then $Y_1 = S^3_{+1}(4_1 \# 4_1^r)$ bounds an equivariant $\ZZ/2\ZZ$-homology ball. By Lemma~\ref{lem:easy-local-equivalence}, this gives a $U$-nontorsion class in $\HFm(Y_1)$ which has Maslov grading zero and is invariant under both $\iota$ and $\tau$. Denote this class by $x_1$. According to Lemma~\ref{lem:cobordism}, we have a $(\tau, \iota)$-local map from $Y_1$ to the large surgery $Y_n = S^3_{n}(4_1 \# 4_1^r)$:
\[
(f_{1,n})_{*} \colon \HFm(Y_1) \rightarrow \HFm(Y_n, \s_0).
\]
This has grading shift $(n-1)/4$. On the other hand, the isomorphism
\[
(\Gamma_{n, 0})_{*} \colon \HFm(Y_n, \s_0) \xrightarrow{\simeq} H_*(A_0(4_1 \# 4_1^r))
\]
has grading shift $-(n-1)/4$ and is both $\tau$- and $\iota$-equivariant. The class $(\Gamma_{n,0} \circ f_{1,n})_{*}(x_1)$ then contradicts Lemma~\ref{lem:towers}, completing the proof.

The argument that $(4_1)_{2,1}$ is nontorsion is similar. Suppose $(4_1)_{2,1}$ had finite order in the concordance group. Then $\#_N(4_1)_{2,1}$ is smoothly slice for some $N > 0$. Taking branched double covers shows that
\[
\Sigma_2(\#_N (4_1)_{2,1}) = \#_N \Sigma_2((4_1)_{2,1}) = \#_N S^3_{+1}(4_1 \# 4_1^r)
\]
bounds an equivariant $\ZZ/2\ZZ$-homology ball. (Here, the connected sum $\#_N S^3_{+1}(4_1 \# 4_1^r)$ is done in the obvious equivariant manner.) It thus suffices to prove that $\CFm(\#_N Y_1)$ admits no local map from the trivial complex. For this, we assume the reader has some familiarity with the language of local equivalence, as discussed in \cite[Section 9]{hmz}. The desired claim follows easily from the fact that $\CFm(Y_1)$ admits a local map \textit{into} the trivial complex; we prove this below in Lemma~\ref{lem:localmap}.

Indeed, let $F \colon \CFm(Y_1) \rightarrow \CFm(S^3)$ be such a map. Dualizing gives a local map $G \colon \CFm(S^3) \rightarrow \CFm(Y_1)^\vee$. If we had a local map from $\CFm(S^3)$ to $\CFm(\#_N Y_1)$, then tensoring this with $\otimes^{N-1} G$ and applying the connected sum formulas of \cite[Theorem 1.1]{hmz} and \cite[Proposition 6.8]{dhm} would give a local map from $\CFm(S^3)$ to $\CFm(Y_1)$, which we already know is a contradiction.
\end{proof}


\begin{lem}
\label{lem:localmap}
There exists a local map from $CF^-(S^3 _{+1}(4_1\# 4_1^r))$ to the trivial complex.
\end{lem}
\begin{proof}
By Table~\ref{table:computation}, the homology $H_*(A_0(4_1 \# 4_1^r))$ is isomorphic to $\ff[U] \oplus \ff^4$. The structure theorem for modules over a principal ideal domain implies that $A_0(4_1 \# 4_1^r)$ is homotopy equivalent to a complex generated over $\ff[U]$ by $\{x, a, b, c, d, e, f, g, h\}$, where $\partial a = Ub$, $\partial c = Ud$, $\partial e = Uf$, and $\partial g = Uh$, and we identify
\[
x = x|x, \quad b = x | d, \quad d = d|x, \quad f = a|d + d|a + b|b + c|c, \quad \text{and} \quad h = d|d.
\]
On this complex, the action of $\iota$ is given by
\begin{align*}
x\mapsto x+b+d+h, \quad & b\mapsto b+h, \quad d\mapsto d+h, \quad f\mapsto f+b+d, \quad h \mapsto h\\
&a\mapsto a+g, \quad c\mapsto c+g, \quad e\mapsto e+a+c, \quad g \mapsto g.
\end{align*}
The action of $\tau$ is given by
\begin{align*}
x\mapsto x, \quad & b\mapsto d, \quad d\mapsto b, \quad f\mapsto f+h, \quad h \mapsto h\\
&a\mapsto c, \quad c\mapsto a, \quad e\mapsto e+g, \quad g \mapsto g.
\end{align*}
The easiest way to see this is to note that in this case, $\iota$ and $\tau$ are determined by their actions on homology for grading reasons (see for example the proof of \cite[Lemma 4.2]{dm}) and then use Table~\ref{table:computation}. Consider the $U$-equivariant map $F:C\rightarrow \ff[U]$, defined by
\[
x\mapsto 1 \quad \text{and} \quad (\text{everything else})\mapsto 0.
\]
It is clear that $F\iota = \iota F$ and $F\tau = \tau F$. Hence $F$ gives a local map from $A_0(4_1 \# 4_1^r)$ to the trivial complex. Precomposing this with the map $f_{1,n}$ in \Cref{lem:cobordism} and using the $\tau$- and $\iota$-equivariant homotopy equivalence $\Gamma_{n, 0}$ gives the desired local map.
\end{proof}


We now turn to the proof of Theorem~\ref{thm:mainthm2}. For this, we again assume the reader has some familiarity with the machinery of local equivalence. Theorem~\ref{thm:mainthm2} will follow quickly from the fact if $\CFK(K_1)$ is $\iota_K$-locally equivalent to $\CFK(K_2)$ in the sense of Zemke \cite[Definition 2.4]{zemke2019connected}, then $\CFK(K_1 \# K_1^r)$ and $\CFK(K_2 \# K_2^r)$ are $(\tau_K, \iota_K)$-locally equivalent in the sense of \cite[Definition 2.13]{dms}. More precisely, recall from the work of Zemke \cite{zemke2019connected} that we have a homomorphism
\[
\CFK_{UV}: \mathcal{C} \rightarrow \mathfrak{K}_\iota, \quad [K] \mapsto [(\CFK(K), \iota_K)]
\]
from the smooth concordance group to the local equivalence group $\mathfrak{K}_\iota$ of $\iota_K$-complexes. (We follow the notation of \cite{dms}, rather than the usual notation $\mathfrak{I}_K$ of \cite{zemke2019connected}.) 
In \cite{dms}, it is shown that we similarly have a homomorphism from Sakuma's strongly invertible concordance group $\smash{\widetilde{\cC}}$ (see \cite{sakuma}) to the group of $(\tau_K, \iota_K)$-complexes:
\[
\CFK_{UV}: \smash{\widetilde{\cC}} \rightarrow \mathfrak{K}_{\tau,\iota}, \quad [K] \mapsto [(\CFK(K), \iota_K, \tau_K)]
\]
In \cite[Section 3]{bi}, it is noted that there is a doubling homomorphism from $\cC$ to $\smash{\widetilde{\cC}}$. 
Here, we produce the algebraic analogue of this doubling map:


\begin{lem}\label{lem:doublingmap}
There exists a map $\mathcal{D} \colon \mathfrak{K}_\iota \rightarrow \mathfrak{K}_{\tau, \iota}$ such that the following square commutes. 
\[
\xymatrix{
\mathcal{C} \ar[rr]^{K\mapsto K\# K^r} \ar[d]_{\CFK_{UV}} & & \widetilde{\cC} \ar[d]^{\CFK_{UV}} \\
\mathfrak{K}_\iota \ar[rr]^{\mathcal{D}} & & \mathfrak{K}_{\tau,\iota}
}
\]
\end{lem}
\begin{proof}
Given an $\iota_K$-complex $C = (C, \iota)$, define $C^r$ to be the complex 
\[
C^r = C\otimes_{\ff[\cU,\cV]}\ff[\cV,\cU].
\]
This means that we interchange the roles of $\cU$ and $\cV$, together with the two components of the bigrading. Precomposing $\iota \colon C \rightarrow C$ with the Sarkar map gives a map $\varsigma \circ \iota \colon C \rightarrow C$. This induces a map 
\[
(\varsigma \circ \iota)^r \colon C^r \rightarrow C^r
\]
which makes $C^r$ into an $\iota_K$-complex. Let
\[
    \mathcal{D}(C) = (C\otimes C^r,\tau_{\otimes},\iota_{\otimes}), 
\]
where
\[
\begin{split}
    \tau_{\otimes} &= (\id \otimes \id +\Psi_C \otimes \Phi_{C^r}) \circ \tau_{\exch} \\
    \iota_{\otimes} &= \varsigma_{C \otimes C^r} \circ (\id \otimes \id +\Psi_C \otimes \Phi_{C^r}) \circ (\iota\otimes (\varsigma \circ \iota)^r)
\end{split}
\]
as in \Cref{thm:swapping}. Here, $\tau_{\exch}$ is defined by 
\[
\tau_{\exch}(x\otimes y)= y\otimes x.
\]
Note that since the identity map $C\rightarrow C^r$ is a skew-isomorphism, $\tau_{\exch}$ is a skew-isomorphism as well.

To show that $\mathcal{D}$ is well-defined, we must show that $\mathcal{D}$ maps $\iota_K$-locally equivalent complexes to $(\tau_K, \iota_K)$-locally equivalent complexes. To prove this, let $f:C_1\rightarrow C_2$ be an $\iota_K$-local map. Then $f$ induces an $\iota_K$-local map $f^r:C^r_1\rightarrow C^r_2$. Consider the map
\[
f\otimes f^r:C_1\otimes C^r_1\rightarrow C_2\otimes C^r_2.
\]
In \cite[Lemma 2.8]{zemke2019connected}, it is shown that $\Phi$ and $\Psi$ commute with all chain maps up to homotopy. It follows that $f\otimes f^r$ homotopy commutes with $\iota_\otimes$. Furthermore, by definition of $f^r$, we have
\[
\tau_{\exch} \circ (f\otimes f^r) = (f\otimes f^r)\circ \tau_{\exch}.
\]
Again using the fact that $\Phi$ and $\Psi$ commute with all chain maps, this shows $f\otimes f^r$ homotopy commutes with $\tau_\otimes$. Thus $f\otimes f^r$ is a $(\tau_K, \iota_K)$-local map, as desired.
\end{proof}
\noindent
Lemma~\ref{lem:doublingmap} should be compared with \cite[Theorem 1.8]{mallick-surgery}, in which the large-surgery actions of $\tau_K$ and $\iota_K$ on $\CFK(K \# K^r)$ are used to define concordance invariants for $K$.

\begin{proof}[Proof of \Cref{thm:mainthm2}]
If $\tau(K) \neq 0$, then it follows from the cabling formula of \cite[Theorem 1]{hom} that $\tau(K_{2,k} \# - T_{2,k}) \neq 0$ and thus that $K_{2,k} \# - T_{2,k}$ is not slice. (Here, $\tau$ is the concordance invariant of \cite{OS-fourball}.) Hence we may assume $\tau(K) = 0$. Under the assumptions of the theorem, it then follows from \cite[Section 8]{hm} that $\CFK(K)$ is $\iota_K$-locally equivalent to $\CFK(4_1)$. By Lemma~\ref{lem:doublingmap}, this implies that $\CFK(K \# K^r)$ is $(\tau_K, \iota_K)$-locally equivalent to $\CFK(4_1 \# 4_1^r)$. Note that if $\CFK(K_1)$ and $\CFK(K_2)$ are $(\tau_K, \iota_K)$-locally equivalent, then $A_0(K_1)$ and $A_0(K_2)$ are easily seen to be $(\tau, \iota)$-locally equivalent.

Taking the branched double cover over $K_{2, k} \# - T_{2, k}$ gives
\[
\Sigma_2(K_{2, k} \# - T_{2, k}) \cong \Sigma_2(K_{2, k}) \# - \Sigma_2(T_{2, k}) \cong S^3_k(K \# K^r) \# - L(k, 1).
\]
Since $L(k, 1)$ is an L-space, the Heegaard Floer homology of the right-hand side (in the self-conjugate $\spinc$-structure) is isomorphic to $\HFm(S^3_k(K \# K^r), \s_0)$ up to an overall grading shift. This isomorphism is $(\tau, \iota)$-equivariant; here we are using \cite[Proposition 6.8]{dhm}. If $k$ is large, then (again up to appropriate grading shift) this latter complex is already $(\tau, \iota)$-equivariantly isomorphic to $A_0(K \# K^r)$. Otherwise, we use a similar cobordism as in Section~\ref{sec:4} to construct a $(\tau, \iota)$-local map from $\CFm(S^3_k(K \# K^r), \s_0)$ to the Heegaard Floer complex of large surgery, and thus into $A_0(K \# K^r)$. Since $A_0(K \# K^r)$ is $(\tau, \iota)$-locally equivalent to $A_0(4_1 \# 4_1^r)$, the rest of the argument is exactly the same as in the proof of Theorem~\ref{thm:mainthm1}.
\end{proof}

\begin{rem}\label{rem:homologyconcordance}
Our proof of Theorem~\ref{thm:mainthm2} shows that $K_{2,k} \# -T_{2,k}$ is not slice in any $\ZZ/2\ZZ$-homology ball. Indeed, if a knot admits a slice disk $D$ in such a ball, then it follows from the proof of \cite[Lemma 4.2]{casson1978slice} that we can form the branched double cover over $D$ and that this double cover is again a $\ZZ/2\ZZ$-homology ball. Hence the same obstruction as discussed in Section~\ref{sec:2} applies. Likewise, the argument of Theorem~\ref{thm:mainthm2} shows that $K_{2,1}$ is not torsion in the $\ZZ/2\ZZ$-homology concordance group.
\end{rem}

\begin{rem}\label{rem:futurework}
For the reader familiar with \cite{dhm}, we emphasize that considering the $\tau$- and $\iota$-actions on $A_0(4_1 \# 4_1^r)$ simultaneously is crucial. Indeed, it is straightforward to check that the actions of $\tau$, $\iota$, and $\tau \circ \iota$ are all individually locally trivial on $A_0(4_1 \# 4_1^r)$. The formalism used in this article is thus a slight enhancement of the setup of \cite{dhm} and \cite[Theorem 1.8]{mallick-surgery}, in which the action of $\iota$ is incorporated only through the consideration of $\tau \circ \iota$. 
\end{rem}

\bibliographystyle{amsalpha}
\bibliography{ref}
\end{document}